\documentclass{amsart}
\usepackage{amssymb,amsfonts, color,epsf}
\usepackage [cmtip,arrow]{xy}
\xyoption{all}
\usepackage {pb-diagram,pb-xy}
\usepackage{enumerate}
\usepackage{subfigure}
\ifx\pdfpageheight\undefined
   \usepackage[dvips,colorlinks=true,linkcolor=blue,citecolor=red,%
      urlcolor=green]{hyperref}
   \usepackage[dvips]{graphicx}
   \makeatletter
   \edef\Gin@extensions{\Gin@extensions,.mps}
   \makeatother
\else
   \usepackage[bookmarksopen=false,pdftex=true,breaklinks=true,%
      backref=page,pagebackref=true,plainpages=false,%
      hyperindex=true,pdfstartview=FitH,colorlinks=true,%
      pdfpagelabels=true,colorlinks=true,linkcolor=blue,%
      citecolor=red,urlcolor=green,hypertexnames=false%
      ]%
   {hyperref}
\fi
\usepackage{bm}

\usepackage{tikz-cd}
\makeatletter
\tikzset{
  column sep/.code=\def\pgfmatrixcolumnsep{\pgf@matrix@xscale*(#1)},
  row sep/.code   =\def\pgfmatrixrowsep{\pgf@matrix@yscale*(#1)},
  matrix xscale/.code=%
    \pgfmathsetmacro\pgf@matrix@xscale{\pgf@matrix@xscale*(#1)},
  matrix yscale/.code=%
    \pgfmathsetmacro\pgf@matrix@yscale{\pgf@matrix@yscale*(#1)},
  matrix scale/.style={/tikz/matrix xscale={#1},/tikz/matrix yscale={#1}}}
\def\pgf@matrix@xscale{1}
\def\pgf@matrix@yscale{1}
\makeatother

\newtheorem{theorem}{Theorem}

\newtheorem{proposition}{Proposition}

\newtheorem*{theorem*}{Theorem}

\theoremstyle{definition}
\newtheorem{definition}{Definition}

\newtheorem{notation}{Notation}

\usepackage{algorithm}
\usepackage{algpseudocode}

\usepackage{tikz}

\algnewcommand\algorithmicinput{\textbf{Input:}}
\algnewcommand\INPUT{\item[\algorithmicinput]}
\algnewcommand\algorithmicoutput{\textbf{Output:}}
\algnewcommand\OUTPUT{\item[\algorithmicoutput]}

\algnewcommand\algorithmiccomplexity{\textbf{Complexity:}}
\algnewcommand\COMPLEXITY{\item[\algorithmiccomplexity]}

\algnewcommand\algorithmicproc{\textbf{Procedure:}}
\algnewcommand\PROCEDURE{\item[\algorithmicproc]}

\newlength{\continueindent}
\usepackage{etoolbox}
\makeatletter
\newcommand*{\ALG@customparshape}{\parshape 2 \leftmargin \linewidth \dimexpr\ALG@tlm+\continueindent\relax \dimexpr\linewidth+\leftmargin-\ALG@tlm-\continueindent\relax}
\apptocmd{\ALG@beginblock}{\ALG@customparshape}{}{\errmessage{failed to patch}}
\makeatother

\theoremstyle{remark}
\newtheorem{remark}{Remark}

\definecolor{DarkBlue}{rgb}{0,0.1,0.55}

\numberwithin{equation}{section}

\newcommand{\eval}{\mathrm{eval}}

\newcommand {\hide}[1]{}

\newcommand {\ZZ} {\mathrm{Zer}}

\newcommand {\RR} {\mathrm{Reali}}
\newcommand {\Z}  {{\mathbb Z}}

\renewcommand {\D}     {\mathrm{D}}
\newcommand {\K}     {\mathrm{K}}

 \newcommand {\N}         {{\mathbb N}}

\newcommand {\A}     {\mathrm{A}}

\newcommand {\sign}        {\mathrm{sign}}

\newcommand {\eps} {{\varepsilon}}
\newcommand {\R} {\mathrm{R}}

\newcommand {\Ext} {\mathrm{Ext}}
\newcommand {\la}   {{\langle}}
\newcommand {\ra}   {{\rangle}}

\newcommand{\udots}{{\mathinner{\mskip1mu\raise1pt\vbox{\kern7pt\hbox{.}}\mskip2mu\raise4pt\hbox{.}\mskip2mu\raise7pt\hbox{.}\mskip1mu}}}

\newcommand{\SIGN}  {\mathrm{SIGN}}

\newcommand {\bit} {\mathrm {bit}}

\hide{
\def\strictsubset{\mathrel{\mathop{\kern 0pt \SRbset}\limits_{\not=}}}

\def\addots{\mathinner{\mkern1mu
\raise1pt\vbox{\kern7pt\hbox{.}}
\mkern2mu\raise4pt\hbox{.}\mkern2mu
\raise7pt\hbox{.}\mkern1mu}}

}

\begin{document}

\title{Quantitative Curve Selection Lemma}

\author{
  Saugata Basu
  \and
  Marie-Fran{\c c}oise Roy
}
\dedicatory{Dedicated to the memory of Professor Masahiro Shiota.}
\thanks{Basu was partially supported by NSF grants CCF-1618918, DMS-1620271 and CCF-1910441.}

\maketitle


\begin{abstract}
  We prove a quantitative version of the curve selection lemma. Denoting by $s,d,k$ bounds the number, the maximum total degree and the number of variables  
of the polynomials describing a semi-algebraic set $S$ and a point $x$ in $\bar S$, we find a semi-algebraic path starting at $x$ and entering in $S$ with a description of degree $(O(d)^{3k+3},O(d)^{k})$ (using a precise definition of the description of a semi-algebraic path and its degree given in the paper).
  As a consequence, we prove that there exists a semi-algebraic path starting at $x$ and entering in $S$,
  such that the degree of the Zariski closure of the image of this path  is bounded by $O(d)^{4k+3}$,
  improving a result in \cite{Kurdyka2014}.
  We also give an algorithm for 
  describing
  the real isolated points of $S$ 
  whose complexity is bounded by
  $s^{2 k+1}d^{O(k)}$
  improving a result in \cite{LeMohabTimo}.
\end{abstract}

\section{Introduction and statement}

The Curve Selection Lemma for semi-algebraic sets is the following result.

We fix a real closed field $\R$ (typically, $\R =\mathbb{R}$).

\begin{theorem}[Curve Selection Lemma]
\label{csl}
  Let $S \subset \R^k$ be a semi-algebraic set and $x \in \bar{S}$. Then there
  exists a positive element $t_0$ of $\R$, and a semi-algebraic path $\varphi$ from
  $[0, t_0)$ to $\R^k$ such that $\varphi (0) = x$ and $\varphi ((0, t_0)) \subset S$.
\end{theorem}

This result, due to Lojasiewicz 
\cite{Loj1,Loj3} (see also \cite{Milnor3}), 
is one of the basic building blocks of
semi-algebraic, semi-analytic  and o-minimal 
geometry 
and it has numerous applications.
In this paper we consider this result from a quantitative point of view in the semi-algebraic case.
The quantitative study of the Curve Selection Lemma started with \cite{Kurdyka2014}.
The aim of this paper is to present a different approach and to improve the 
result.

We sketch one of the proofs of Theorem \ref{csl}, coming from
\cite[Theorem 3.19]{BPRbook2}, which is relevant for this paper. 

The proof uses properties of the extension of a semi-algebraic set $S \subset \R^k$ to a real closed field $\R'$ containing $\R$. By definition  $\Ext (S, \R' )$ is the semi-algebraic set in $\R^k$ defined as the realization in $\R'^k$ of a quantifier-free formula defining $S$. By 
the Tarski-Seidenberg principle
(more precisely \cite[Proposition 2.87]{BPRbook2}),  $\Ext (S, \R' )$ is well defined and does not depend on the choice of
the quantifier-free formula defining $S$.

It also uses the description of the real closure  $\R\langle \eps \rangle$ of the field $\R(\eps)$ of rational functions in the variable $\eps$ (equipped with the order $0_+$ defined by $0<\eps<r$ for any positive $r\in \R$,
so that $\eps$ is a positive infinitesimal), as the field of germs of continuous semi-algebraic functions 
to the right of the origin, 
where $\eps$ is the germ of the identity function
\cite[Proposition 3.13]{BPRbook2}. 

We are now ready for the sketch of the proof. Since $x\in \bar{S}$, there exists $r>0$ such that
the sphere of center $x$ and radius $t$ intersects $S$ for any $t\in (0,r)$. Hence,  
by Tarski-Seidenberg principle 
\cite[Theorem 2.80]{BPRbook2}, the sphere of center $x$ and radius $\eps$ intersects $\Ext(S,\R\langle \eps \rangle)$ in a point $y_\eps$, which is infinitesimally close to $x$,
since $\|y_\eps-x\|=\eps$ is smaller than any positive element of $\R$. 
The $k$ coordinates of $y_\eps$ 
are germs of semi-algebraic continuous functions, 
so
there exists $t_0>0$ and a $k$-tuple of semi-algebraic continuous functions $\varphi$ defined on $[0,t_0)$ such that $\varphi(0)=x$ and $\varphi(t)\in S$  for $t\in(0,t_0)$. Note that, denoting $\varphi=\Ext(\varphi,\R\la \eps\ra)$, $\varphi(\eps)=y_\eps$.
For more details 
see \cite[Theorem 3.19]{BPRbook2}.

The proof of our quantitative statement  is following the same strategy, extracting explicit bounds on the description of the path 
from the use of
algorithms from \cite{BPRbook2,BPRbookonline}.

Before stating and proving our main result we need to introduce a few notions.

\begin{definition}
  Let $\mathcal{P}$ be a set of $s$ polynomials of
  total
  degree bounded by $d$ in
  $k$ variables with coefficients in an ordered domain $\D$ with fraction field $\K$ contained in a real closed field $\R$. The set $S \subset \R^k$ is a $\mathcal{P}$-semi-algebraic set
  if there exists a Boolean combination $\Phi$ of atoms of the form $P > 0$,
  $P < 0$, $P = 0$, 
  with $P \in \mathcal{P}$
  such that
  \[ S = \{ x \in \R^k \mid \Phi (x) \} . \]
\end{definition}

In order to give quantitative results about the description of a
semi-algebraic path we use the following definitions, making precise how we describe points and paths.
We use Rational Univariate Representations \cite{fab} as well as Thom encodings \cite{CR}, with the notation and terminology from \cite{BPRbook2}.

\begin{definition}
We say that the point $x\in \R^k$ is associated to 
the real $k$-univariate representation
\cite[page 465]{BPRbook2}
$$\xi = \left((f_\xi,g_{\xi,0},\ldots,g_{\xi,k}),\tau_\xi\right),$$ 
where $f_\xi,g_{\xi,i} \in \D[A]$, $f_\xi$ monic, $\deg_A(g_{\xi,i})<\deg_A(f_\xi)$,
and $\tau_\xi$ is a Thom encoding \cite[Definition 2.29]{BPRbook2} of a real root $a_\xi$ of $f_\xi$
 if
 $g_{\xi,0}(a_\xi)\not=0$ and $$x_i=\frac{g_{\xi,i}(a_\xi)}{g_{\xi,0}(a_\xi)}, 1 \leq i \leq k.$$
The degree $d'$ of $\xi$ 
 is the degree of $f_\xi$ with respect to $A$.

We denote by $\D_\xi$ the ring $\D[A]/(f_\xi)$, and define
the evaluation homomorphism  
$\eval_\xi: \D_\xi \rightarrow \R$ 
by
\[
\eval_\xi(\bar{g})=g(a_\xi),
\]
and the sign function
$\sign_\xi: \D_\xi \rightarrow \{-1,0,1\}$  by
\[
\sign_\xi(\bar{g})=\sign(g(a_\xi)), 
\]
denoting by $\bar{g}$ the image of $g \in \D[A]$ in $D_\xi$ under the canonical surjection.

Note that if $g\in (f_\xi)$, 
$g(a_\xi) = 0$ and  $\sign(g(a_\xi)) = 0$.
As a consequence 
$\eval_\xi$ and 
$\sign_\xi$ are well defined.
  
However, note that $\D_\xi$ is not always a domain, since $f_\xi$ is not necessarily irreducible or even reduced.
Consequently, the homomorphism $\eval_\xi$ is not necessarily injective, since
an element $g \in \D[A]$ such that  $g(a_\xi)=0$ does not necessarily belong to the ideal $(f_\xi)$, so that it may happen that $\bar{g} \neq 0$ in $\D_\xi$, but $\eval_\xi(\bar{g}) = 0$.
Also, note that the function $\sign_\xi$ may not define a total order on $\D_\xi$, since there may
 exist
  elements $g \in \D[A]$ such that $\bar{g}\neq 0$ and $\sign_\xi(g)=0$.

 In the special case where the coordinates of $x$ are $x_i=(a_i/b)$, belonging to $\K$, with $(b,a_1,\ldots,a_k)\in \D^{k+1}$, $x$ is associated to the  real
 univariate representation 
 $$\xi=((A,A-b,A-a_1,\ldots,A-a_k),1).$$ 
It is of degree $1$, and $\D_\xi=\D$.
\end{definition}

\begin{definition}
  A \emph{description of a semi-algebraic path} $\varphi$ from $[0, t_0)$ to $\R^k$, with $\varphi(0) = x$, associated to $\xi$, is given by
  $$v = \left((f_v, g_{v,0},\ldots,g_{v,k}),\tau_v\right),$$
with $f_v,g_{v,i} \in \D_\xi[T,U]$, and $\tau_v$ a Thom encoding of a real root $u(t)$ of 
$\eval_\xi(f_v(t,u(t)))$
for all small enough $0 < t \le t_0$,
such that $\varphi$ from $[0, t_0)$ to $\R^k$ defined by
$$\begin{array}{rcl}
 \varphi(0)&= &x,\\
 \varphi(t)&= &\left(\frac{\eval_\xi(g_1(t,u(t)))}{\eval_\xi(g_0(t,u(t)))},\ldots,\frac{\eval_\xi(g_k(t,u(t)))}{\eval_\xi(g_0(t,u(t)))}\right)
 \end{array}$$ 
 is continuous at $0$.

The degree of the description $v$ is the pair $(d_{v,T},d_{v,U})$, where
$d_{v,T}$  is the maximum of the degrees  $$\deg_T(f_v), \deg_T(g_{v,0}),\ldots, \deg_T(g_{v,k})$$ and $d_{v,U}$ is the maximum of the degrees 
$$\deg_U(f_v), \deg_U(g_{v,0}),\ldots, \deg_U(g_{v,k}).$$
\end{definition}

The main result of the paper is the following.

\begin{theorem}[Quantitative Curve Selection Lemma]
\label{qcsl}
  Let 
   $\mathcal{P} \subset \D[X_1,\ldots,X_k]$ be a finite set of $s$ polynomials of maximum 
   total
   degree $d$,
  $S$ a $\mathcal{P}$-semi-algebraic set,  and
   $x \in \bar{S}$ 
  associated to
  $\xi$ of degree $d'$. There 
  exists $t_0\in \R, t_0 > 0$,
  a semi-algebraic path $\varphi: [0,t_0) \rightarrow \R^k$ such that,  
  $\varphi(0) = x$,
  and $\varphi ((0,
  t_0)) \subset S$ 
with a description $v$ of degree
    \[
    (2 d N N'^2,N) \in (O(d)^{3k+3},O(d)^{k})
    \] 
    with
    \begin{eqnarray}
    \label{eqn:NN'}
    \nonumber
     N &=& (2d +6)  (2d +5)^{k-1}, \\
     N' &=& 5 (k(2d+4)+2) N.
    \end{eqnarray}
  Moreover the description $v$ can be computed with $d'^{O(1)} s^k d^{O(k)}$ arithmetic operations in $\D$.
\end{theorem}

We prove Theorem \ref{qcsl} as a corollary of a  more general result.

We denote by $S(x,t)$ and $B(x,t)$  the sphere and the ball  (respectively)  in $\R^k$ with center $x\in \R^k$ 
and radius $t\in \R$. We also denote by $S(x,\eps)$ and $B(x,\eps)$ 
the sphere and the ball  (respectively)  in $\R\la\eps\ra^k$ with center $x\in \R\la \eps\ra^k$ 
and radius $\eps$.

For any finite family of polynomials $\mathcal{P} \subset \R[X_1,\ldots,X_k]$, a sign condition on $\mathcal{P}$ is an element
$\{0,1,-1\}^{\mathcal{P}}$.  If $Z \subset \R^k$ is an any semi-algebraic set and  $\sigma \in \{0,1,-1\}^{\mathcal{P}}$ a sign condition
on $\mathcal{P}$,  we denote by 
\[
\RR(\sigma,Z) = \{ x\in Z \;\mid\; \sign(P(x)) = \sigma(P), P \in \mathcal{P}\},
\]
and call $\RR(\sigma,Z)$ the realization of $\sigma$ on $Z$. A sign condition $\sigma$ is called \emph{realizable} if 
$\RR(\sigma,\R^k) \neq \emptyset$.  

\begin{notation}
\label{not:sign_n}
We denote by 
$\SIGN_n (x,\mathcal{P} )$ the set of realizable sign conditions 
of $\mathcal{P}$ at the neighbourhood of $x$, i.e.
the set of sign conditions $\sigma \in \{0,1,-1\}^{\mathcal{P}}$ such that
there exists $r$ a positive element of $\R$ such that for every $t \in \R$, such that $0<t<r$
$$\RR(\sigma,\R^k)\cap S(x,t)\not= \emptyset,$$
or, equivalently, by \cite[Theorem 2.98 and Proposition 3.20]{BPRbook2},
$$\RR(\sigma,\R\langle \eps \rangle^k)\cap S(x,\eps)\not= \emptyset.$$
\end{notation}

\begin{theorem}[Construction of little paths]
\label{qcsltotcc}
  Let $\mathcal{P} \subset \R[X_1,\ldots,X_k]$ be a
  finite set of $s$ polynomials of maximum 
  total
  degree $d$,  and  $x \in \R^k$ associated to $\xi$ of degree $d'$.
   Then, there exist 
    $(a_0,b_0) \in \D_\xi^2, \sign_\xi(a_0)=\sign_\xi(b_0)=1$, and
for every  
$\sigma\in \SIGN_n( x,\mathcal{P} )$,  
 a set $\Phi_\sigma$ of 
non-constant
semi-algebraic paths \[
\varphi: \left[0,t_0=\frac{\eval_\xi(a_0)}{\eval_\xi(b_0)}\right) \rightarrow \R^k
\]
such that $\varphi(0) = x$, and
 $\varphi((0,t_0))\subset \RR(\sigma,B(x,t_0))$.
  The set of $\varphi(\eps), \varphi \in \Phi_\sigma $ has a non-empty intersection
 with  every semi-algebraically connected component of  
 \[
 \RR(\sigma,\R\langle \eps \rangle^k)\cap S(x,\eps).
 \]
 The set, 
 $\mathcal{W}_\sigma$, 
   of descriptions of the elements of $\Phi_\sigma$
are  of degrees bounded
  \[(2 d N N'^2,N) \in (O(d)^{3k+3},O(d)^{k}),\] with
     \[ N= (2d +6)  (2d +5) ^{k-1} , \]
       \[N'= 5 (k(2d+4)+2) N.
       \] 
  When $\D=\Z$ and $x=0$, 
  and the bitsizes of the coefficients of the polynomials in $\mathcal{P}$ are bounded by $\tau$, 
  the bitsizes of the coefficients  appearing in the descriptions $\mathcal{W}_{\sigma}$,  and the numerator and denominator of $t_0$  
have bitsizes bounded by   $ \tau (d+k)^k O(d)^{3k+2}$.

Moreover,  the sets  of descriptions $\mathcal{W}_{\sigma}, \sigma\in \SIGN_n( x,\mathcal{P})$, can be computed with 
$d'^{O(1)}s^{k+1}  d^{O(k)}$ 
arithmetic operations in $\D$.
\end{theorem}

\section{Proving theorem \ref{qcsltotcc}}

Our strategy is to find points at infinitesimal distance from $x$ inside realizable sign conditions on $\mathcal P$ with good degree bounds on their definition and to use these points with coordinates in $\R\langle \eps \rangle$, the real closed field of germ of semi-algebraic functions,  to obtain the description of semi-algebraic paths.
The algorithm is very similar to \cite[Algorithm 13.1 (Computing realizable sign conditions)]{BPRbook2} except in the last part.

\begin{notation}\label{13:not:P0}
Define 
        $$P_0= (g_{\xi,0}X_{1}-g_{\xi,1})^{2} + \cdots +(g_{\xi,0} X_{k}-g_{\xi,k})^{2} -\eps^2,$$
        where $x$ is associated to
        $$\xi = \left((f_\xi,g_{\xi,0},\ldots,g_{\xi,k}),\tau_\xi\right),$$  
        and note that
        $$\ZZ(P_0,\R\la \eps\ra^k) = S(x,\eps).$$
    \end{notation}

\begin{notation}\label{13:not:bigh}
Let $\mathcal{P}=\{P_1,\ldots,P_s\} \subset \R[X_1,\ldots,X_k]$ 
  of maximum  
  total
  degree $d$.
Define $\hat d$ to be the smallest even number $>d$.
  Define
  \begin{eqnarray*}
       H_{k} (\hat d,i)& =&1+ \sum_{1 \leq j \leq k} i^{j} X_{j}^{\hat d} ,\\
       H_{k}^{h} (\hat d,i) &=&X_{0}^{\hat d} + \sum_{1 \leq j \leq k} i^{j} X_{j}^{\hat d} .
     \end{eqnarray*}
  Note that 
  $H_{k} (\hat d,i) (x) >0$ for every $x \in \R^{k}$. 
  
  Let $\delta$ and $\gamma$ be two new variables.
  
Define for $1 \leq i \leq s$,
    \begin{eqnarray*}
      P^{\star}_{i} & = & \{ (1- \delta ) P_{i} + \delta H_{k} (\hat d ,i) , (1-
      \delta ) P_{i} - \delta H_{k} (\hat d ,i) ,\\
      &  & (1- \delta ) P_{i} + \delta \gamma H_{k} (\hat d ,i) , (1- \delta )
      P_{i} - \delta \gamma H_{k} (\hat d ,i) \}. 
    \end{eqnarray*}
    \end{notation}

We also need  the following notation.

\begin{notation}
 [Limit of a Puiseux series]
  \label{12:not:lim}
  For any real closed field $\R$,
  the real closure  $\R\langle \eps \rangle$ of the field $\R( \eps)$ equipped with the order defined by $0<\eps<r$ for any positive $r\in \R$, can be described as the field of algebraic Puiseux series in $\eps$. The $\lim_\eps$ function associates to an element $y_\eps$ of  $\R\langle \eps \rangle$ which is bounded over $\R$ its limit $\lim_\eps(y_\eps)\in \R$ defined by substituting $0$ to $\eps$ in the Puiseux series $y_\eps$.
\end{notation}

  We finally denote $\R\la\eps,\delta\ra  = \R\la\eps\ra\la\delta\ra$, $\R\la\delta,\gamma\ra  = \R\la\delta\ra\la\gamma\ra$ and $\R\la\eps,\delta,\gamma\ra = \R\la\eps,\delta\ra\la\gamma\ra$. As a consequence, $\delta$ is an infinitesimal smaller than any positive elements of $\R\la\eps\ra$ and $\gamma$ is an infinitesimal smaller than any positive elements of 
  $\R\la\eps,\delta\ra$.

\begin{proposition}\label{prop:Hk}
\begin{enumerate}[(a)]
\item
\label{itemlabel:prop:Hk:1}
For every $i_1<\ldots <i_{k+1}$ and $Q_{i_j}\in  P^{\star}_{i_j}$
$$\ZZ(\{Q_{i_1},\ldots,Q_{i_{k+1}}\},\R\langle \delta,\gamma \rangle^k)=\emptyset.$$
\item 
\label{itemlabel:prop:Hk:2}
For every semi-algebraically connected component $C\subset S(x,\eps)$ of the realization of a sign condition on $\mathcal{P}$ restricted to $S(x,\eps)$, there exist $\ell \le k$, $i_1<\ldots <i_{\ell}$, $Q_{i_j}\in  P^{\star}_{i_j}$
and a semi-algebraically connected component $D$ of 
\[
\ZZ(\{P_0,Q_{i_1},\ldots,Q_{i_{k+1}}\},\R\langle \eps,\delta,\gamma \rangle^k)
\] 
such that $\lim_\gamma(D)\subset 
\Ext(C,\R\la\eps,\delta\ra)
$.
\end{enumerate}
\end{proposition}

The proof of Proposition \ref{prop:Hk} uses the following result that appears in \cite[Proposition 13.1]{BPRbook2}.
\begin{proposition}
  \label{13:pro:algebraic fill}Let $D \subset \R^{k}$ be a non-empty
  semi-algebraically connected component of a basic closed semi-algebraic set
  defined by
  \[ P_{1} = \cdots =P_{\ell} =0,P_{\ell +1} \geq 0, \cdots ,P_{s} \geq 0. \]
  There exists  $\{i_{1} , \ldots ,i_{m} \} \subset \{\ell +1, \ldots ,s\}$
  such that the algebraic set $W$ defined by equations
  \[ P_{1} = \cdots =P_{\ell} =P_{i_{1}} = \cdots P_{i_{m}} =0, \]
  has
  a semi-algebraically connected component $D'$ 
   contained in $D$.
\end{proposition}

\begin{proof}[Proof of Proposition \ref{prop:Hk}]
Part \eqref{itemlabel:prop:Hk:1} follows from Proposition 13.6 in \cite{BPRbook2}.

We now prove Part \eqref{itemlabel:prop:Hk:2}. 
Let (without loss of generality), $\sigma$ be the sign condition $P_1 = \cdots = P_\ell =0, P_{\ell+1} > 0, \cdots, P_s > 0$, and
$C\subset S(x,\eps)$, the realization of $\sigma$ restricted to $S(x,\eps)$.
Consider two points $x$ and $y$ in $C$. There is a semi-algebraic path $\theta$ from $x$ to $y$ inside
  $C$. Since the image of $\theta$ is closed and bounded, the semi-algebraic and continuous
  function $\min_{\ell +1 \le i \le s} (P_{i} )$ has a strictly positive
  minimum on the image of $\theta$. The extension of the path $\theta$ to $\R \la \eps,\delta,\gamma\ra$
  is thus entirely contained inside the subset $T$ 
  of $\R \la \eps, \delta,\gamma \ra^{k}$
  defined by
  $$\displaylines{
  P_0 = 0, \cr
  -\delta \gamma H_k(\hat d,1) \leq (1-\delta)P_{1} \leq 
 \delta \gamma H_k(\hat d,1),\cr
   \vdots \cr
  -\delta \gamma H_k(\hat d,\ell) \leq (1-\delta)P_{\ell} \leq 
 \delta \gamma H_k(\hat d,\ell),\cr
  (1-\delta)P_{\ell+1} - \delta H_k(\hat d,\ell+1) \geq 0,\cr
   \vdots \cr
   (1-\delta)P_{s} - \delta H_k(\hat d,s) \geq 0.
   }
   $$
   
  Thus, there is one non-empty semi-algebraically connected component
  $D'$ of $T$ containing  $\Ext(C,\R\la \eps, \delta,\gamma\ra)$.  Now applying Proposition \ref{13:pro:algebraic
  fill} to $D'$ and $T$, we get a semi-algebraically connected component
  $D$ of some algebraic set defined by
  $$\displaylines{
  P_0 = 0, \cr
   (1-\delta)P_{i_1} + s_{i_1} \delta \gamma
   H_k(\hat d,i_1) = 0, \cr
   \vdots 
   \cr
  (1-\delta)P_{i_m} + s_{i_m}  \delta \gamma
   H_k(\hat d,i_m) = 0, \cr
  (1-\delta)P_{j_1} - \delta  H_k(\hat d,j_1) = 0, \cr
  \vdots
  \cr
  (1-\delta)P_{j_n}  - \delta H_k(\hat d,j_n) =0,
  }
   $$
(with
$1\leq i_1 < \cdots < i_m \leq \ell, \ell+1 \leq j_1 < \cdots < j_n \leq s$ and  $s_{i_j}\in \{-1,1\}$)
contained in $D'$, and $\lim_\gamma D$ is contained in $\Ext(C,\R \la \eps,\delta\ra)$.
\end{proof}

Before describing the algorithm we need some more notation.
\begin{notation}
\label{not:homogenization}
Let $\A$ be a ring and 
\[
Q = \sum_{\alpha =(a_1,\ldots,\alpha_k) \in \N^k} a_\alpha X_1^{\alpha_1}\cdots X_k^{\alpha_k}   \in \A[X_1,\ldots,X_k].
\] 
We denote by $Q^h \in A[X_0,\ldots,X_k]$ (the homogenization of $Q$) the polynomial
\[
Q^h := \sum_\alpha X_0^{e - \sum_i \alpha_i} X_1^{\alpha_1} \cdots X_k^{\alpha_k}, 
\]
where $e$  is the least even number not less than the total degree of $\deg(Q)$.
\end{notation}

\begin{notation}
 [Substituting a $k$-univariate representation in a polynomial]
  \label{12:not:subsrur}
  Let 
  \[
  u= (f,g) \in \D[T]^{k+2} , g= (g_{0} , \ldots 
  ,g_{k})
  \] 
  be a $k$-univariate representation and $Q \in \D[X_{1} , \ldots, X_{k} ]$. 
  Set
  \label{12:not:subsrur2}
  \begin{equation}
    \label{12:eq:P} Q_{u} =
    Q^h(g_0,g_1,\ldots,g_k).
  \end{equation}
\end{notation}

We will use the following definitions in Algorithm \ref{alg:1} (Computing Local Semi-algebraic Paths).

\begin{definition}
\label{def:xi}
Given a point $x$ associated to $\xi$ with coefficients in $\D$:
\begin{enumerate}[(a)]
\item  
\label{itemlabel:def:xi:a}
the $\eval_\xi$ homomorphism from $\D_\xi \mapsto \R$ is extended to 
$\D_\xi[\eps] \mapsto \R\la \eps\ra$  (resp.  $\D_\xi[\eps,\delta] \mapsto \R\la \eps, \delta\ra$, $\D_\xi[\eps,\delta,\gamma] \mapsto \R\la \eps, \delta,\gamma\ra$);
\item 
\label{itemlabel:def:xi:b}
the $\sign_\xi$ function from $\D_\xi \mapsto \{-1,0,1\}$ is extended to $\D_\xi[\eps]\mapsto \{-1,0,1\}$ (resp.  $\D_\xi[\eps,\delta] \mapsto \{-1,0,1\}$, $\D_\xi[\eps,\delta,\gamma]\mapsto \{-1,0,1\}$) using the fact that $\eps$ (resp $\delta$, $\gamma$) is smaller than any positive element in $\R$  (resp. $\R\la \eps \ra$,  $\R\la \eps,\delta \ra$);
\item 
\label{itemlabel:def:xi:c}
$\lim_\gamma: \D_\xi[\eps,\delta,\gamma] \mapsto \D_\xi[\eps,\delta]$ is defined by removing
extraneous
 powers of $\gamma$ and replacing $\gamma$ by $0$;
\item 
\label{itemlabel:def:xi:d}
the minimum $\min(a,b)$ of two elements in  $\D_\xi$ (resp. $\D_\xi[\eps]$, $\D_\xi[\eps,\delta]$, $\D_\xi[\eps,\delta,\gamma]$) is equal to $b$ if $\sign_\xi(b-a)=1$ and $a$ otherwise.
Similarly,
the maximum $\max(a,b)$ of two elements in  $\D_\xi$ (resp. $\D_\xi[\eps]$, $\D_\xi[\eps,\delta]$, $\D_\xi[\eps,\delta,\gamma]$) is equal to $b$ if $\sign_\xi(a-b)=1$ and $a$ otherwise.
\end{enumerate}
\end{definition}

\begin{remark}[About the sign determination algorithm]
\label{rem:xi}
In Algorithm \ref{alg:1} we will use  \cite[Algorithm 10.13
    (Univariate Sign Determination)]{BPRbook2}. The input to this algorithm 
    requires in principle that the
    coefficients of the input polynomials belong to a domain. 
    However, we are going to use
    input polynomials with coefficients in the ring $\D_\xi[\eps,\delta]$ which might not be a domain. 
    An examination of  \cite[Algorithm 10.13
    (Univariate Sign Determination)]{BPRbook2} yields that the ring of coefficients is involved only for Tarksi-queries determinations, which are computed using signs of subresultant coefficients,
     i.e. polynomial expressions in the coefficients of the input polynomials. These Tarksi-queries determinations can be performed in the ring $D_\xi[\eps,\delta]$ equipped with 
     its
     sign function  $\sign_\xi$  (cf. Definition~\ref{def:xi} Part~\eqref{itemlabel:def:xi:b})
and the correctness of the algorithm with input polynomials with coefficients in $\D_\xi[\eps,\delta]$ follows. 
\end{remark}

\begin{notation}
  [Cauchy bound] 
  \label{10:not:bigcauchy2}
  Let $ P=c_{p} X^{p} + \cdots +c_{q} X^{q} \in \D_\xi[\eps]$ with 
  $p>q, \eval_\xi(c_{q} c_{p})  \neq 0$.  We
  denote
  \begin{eqnarray}
    a(P) &=& c_q^2, \\
    b(P) &=& (p+1) \cdot \sum_{q \leq i \leq p} c_i^2.
  \end{eqnarray}
\end{notation}

It is easy to prove (see  \cite[Lemma 10.7]{BPRbook2})
that the absolute value of any non-zero root of $P$ in $\R\la\eps\ra$
  is bigger than 
  $\frac{\eval_\xi(a(P))}{\eval_\xi(b(P))}$.
  
  \begin{algorithm}[H]
\caption{(Computing Local Semi-algebraic Paths)}
\label{alg:1}
\begin{algorithmic}[1]
\INPUT
\Statex{
\begin{enumerate}[(a)]
\item
a point $x$ associated to $\xi$  of degree $d'$;
\item
a set of $s$ polynomials,
  \[ \mathcal{P} = \{P_{1}, \ldots, P_{s} \} \subset \D [X_{1} , \ldots, X_{k}],
   \]
  of degrees bounded by $d$.
\end{enumerate}
 }

\algstore{myalg}
\end{algorithmic}
\end{algorithm}
 
\begin{algorithm}[H]
\begin{algorithmic}[1]
\algrestore{myalg}

 \OUTPUT
 \Statex{
 \begin{enumerate}[(a)]
 \item
 $a_0,b_0 \in \D_\xi^2$ with $\sign_\xi(a_0)=\sign_\xi(b_0)=1$;
 \item
 for each $\sigma \in \SIGN_n(x,\mathcal{P})$ (cf. Notation~\ref{not:sign_n}), 
 a set, $\mathcal{W}_\sigma$, consisting of descriptions 
of  a set, $\Phi_\sigma$, of  non-constant  semi-algebraic paths  
 \[
 \varphi:\left[0,t_0=\frac{\eval_\xi(a_0)}{\eval_\xi(b_0)}\right) \rightarrow \R^k,
 \]
 satisfying 
 \[
 \varphi(0) = x, \varphi((0,t_0))\subset \RR(\sigma,B(x,t_0)),
 \]
and with the property that the set 
\[
\{\varphi(\eps)\mid \varphi \in \Phi_\sigma\}
\]
has a non-empty intersection with every semi-algebraically connected component of  
$\RR(\sigma,\R\langle \eps \rangle^k)\cap S(x,\eps)$.
\end{enumerate}
 }

 \COMPLEXITY
 \Statex{
 \begin{enumerate}[(a)]
 \item
 The number of arithmetic operations in $\D$ is bounded by
 $ d'^{O (1)}s^{k+1} d^{O (k)}$.
 \item
  The degrees of the representations with respect to $T,U$ are bounded by 
  \[
  (2NN'^2,N) \in (O(d)^{3k+3}, O(d)^k),
  \]
  where
  \begin{eqnarray*}
  N &=& (2d +6)  (2d +5) ^{k-1}, \\
  N' &=& 5 (k(2d+4)+2) N.
  \end{eqnarray*}
 \end{enumerate}
 }

 \PROCEDURE
 \State{$\mathcal{U} \gets \emptyset$.}
\For {every set of $j \le k$ polynomials $Q_{i_{1}} \in
    P_{i_{1}}^{\star} , \ldots ,Q_{i_{j}} \in P_{i_{j}}^{\star}$  (Using Notation \ref{13:not:bigh})} 
    	\State{ 
    Apply 
    \cite[Algorithm 12.65]{BPRbookonline} 
    (Parametrized Bounded Algebraic Sampling) with input $Q$, 
    parameters $\eps,\delta,\gamma$, and the ring $\D_\xi$,  where 
      \begin{eqnarray*}
        Q & = & P_0^2+ Q_{i_{1}}^{2} + \cdots +Q_{i_{j}}^{2},
        \end{eqnarray*}
        and output parametrized $k$-univariate representations contained in $\D_\xi[\eps,\delta,\gamma][U]^{k+2}$.
        }
        \State{
     Apply $\lim_\gamma$ (cf. Definition~\ref{def:xi}, Part~\eqref{itemlabel:def:xi:c}) to the 
     parametrized $k$-univariate representations output in the previous step, to obtain  
     parametrized $k$-univariate representations 
        $$(f(U),g_0(U),g_1(U),\ldots,g_k(U))\in \D_\xi[\eps,\delta][U]^{k+2},$$
     and add these parametrized $k$-univariate representations to $\mathcal{U}$.
             }
             \label{alg:line:4}
\EndFor

\For {each $ u \in \mathcal{U}, u = (f(U),g_{0} (U), \ldots ,g_{k} (U)) \in \D_\xi[ \eps , \delta ]
         [U]^{k+2}$}
 \State{
   $\mathcal{P}_{u} \gets \{P_{u}\mid P \in \mathcal{P}\}$.
   }
\State{
         Use   \cite[Algorithm 11.19 (Restricted Elimination)]{BPRbook2} with variable $U$, input $f$ and $\mathcal{P}_{u}$ and
denote  $\mathcal{A}_{u}$ its output. }
\label{alg:line:8}
\EndFor
\State{
$\mathcal{A} \gets \bigcup_{u\in \mathcal{U}}
    \mathcal{A}_{u}
    \subset \D_\xi [\eps,\delta]$.
	}

\State{Using 
    Notation \ref{10:not:bigcauchy2}
     and considering elements of $\mathcal{A}$ as polynomials in $\delta$, take
    $\lambda  \in \mathcal{A}$, such that
    $\sign_\xi(a(\lambda')b(\lambda) - a(\lambda) b(\lambda')) \geq 0$   for all  $\lambda'\in \mathcal{A}$.
}
\label{alg:line:11}

\algstore{myalg}
\end{algorithmic}
\end{algorithm}
 
\begin{algorithm}[H]
\begin{algorithmic}[1]
\algrestore{myalg}

\For {each $u\in \mathcal{U}, u= (f(U),g_{0} (U), \ldots ,g_{k} (U)) \in \D_\xi[ \eps , \delta ]
         [U]^{k+2},$}
  
	\State{ 
         compute the Thom encodings of the roots of $f$ and the signs of $P_u \in \mathcal{P}_u$ at the points associated to
    the 
    $k$-univariate representations in $\mathcal{U}$, using \cite[Algorithm 10.13
    (Univariate Sign Determination)]{BPRbook2}  (cf. Remark~\ref{rem:xi}) with input $f$
    and its derivatives,  and the $P_{u}$, $P \in \mathcal{P}$.
    }
    \label{alg:line:13}
\EndFor

    \State{
     Obtain $\SIGN_n(x,\mathcal{P})$ and 
    for $\sigma \in \SIGN_n(x,\mathcal{P})$, 
    let $\mathcal{V}_\sigma$ be the set of pairs $(u,\tau)$, where $u = (f(U),g_{0} (U), \ldots ,g_{k} (U)) \in \mathcal{U}$, $\tau$ the Thom encoding of a root   $\alpha$ of $\eval_\xi(f)$,
    such that for each $P \in \mathcal{P}$,  the sign  of $P_u$ at the  root $\alpha$ is $\sigma(P)$.
    }

      \For {$(u, \tau) \in V_\sigma$, with $u = (f,g_0,\ldots,g_k) \in \D_\xi[\eps,\delta,U]^{k+2}$}
      \State{
       $u^h  \gets (f^h,g_0^h, \ldots,g_k^h) \in \D_\xi[\eps,\delta_0,\delta,U] ^{k+2}$ (where
       $f^h,g_0^h, \ldots,g_k^h \in \D_\xi[\eps,\delta_0,\delta,U]$ are 
       the homogenizations of $f, g_0,\ldots,g_k \in \D_\xi[\eps,\delta,U]$ with respect to the variable $\delta$).
       }
       \State{
       $
       \mathcal{W}_\sigma \gets \{(u^h(T,b(h),a(h),U),\tau) \mid (u,\tau) \in V_\sigma \}.
       $
       }
      \EndFor

\State{Using 
   Notation \ref{10:not:bigcauchy2} and \cite[Algorithm 11.19 (Restricted Elimination)]{BPRbook2}
    with 
    variable $\delta$, input $h$ and  $\mathcal{A}\setminus \{\lambda\})$,  compute its output
   $\mathcal{B}\subset \D_\xi [ \eps ]$,
    and let
    $$\mathcal{B'}=\mathcal{B} \cup \{a(\lambda) b(\lambda')-a(\lambda') b(\lambda),\lambda'\in \mathcal{A}\setminus \{\lambda\}\}.$$
    }
    
\State{
     Using 
     Notation \ref{10:not:bigcauchy2} take  
     $a_0 = \min_{g\in \mathcal{B'}} a (g)$,  and $b_0= \max_{g\in \mathcal{B'}} b(g)$ 
     (cf. Definition~\ref{def:xi}, Part~\eqref{itemlabel:def:xi:d}).
}
\label{alg:line:21}

\end{algorithmic}
\end{algorithm}

\begin{proof}[Proof of correctness]
As a consequence of Part \eqref{itemlabel:prop:Hk:2} of Proposition \ref{prop:Hk}, 
we have that for every semi-algebraically connected component $D$ of $\RR(\sigma,\R\la\eps\ra^k) \cap S(x,\eps)$, where
$\sigma \in \SIGN_n(x,\mathcal{P})$, there exists a semi-algebraically connected component $D'$
of an algebraic set  
\[
\ZZ(P_0^2+ Q_{i_{1}}^{2} + \cdots +Q_{i_{j}}^{2}, \R\la\eps,\delta,\gamma\ra^k) \subset \Ext(S(x,\eps), \R\la\eps,\delta,\gamma\ra)
\]
with $Q_{i_{1}} \in
    P_{i_{1}}^{\star} , \ldots ,Q_{i_{j}} \in P_{i_{j}}^{\star}$, such that 
    $\lim_\gamma(D')$ is contained in 
    $\Ext(D,\R\langle \eps,\delta\rangle) 
    $, and 
    moreover, using Part \eqref{itemlabel:prop:Hk:1} of Proposition \ref{prop:Hk}, 
    $j \leq k$.
Also, notice that since 
\[
Z(P_0,\R\la \eps,\delta,\gamma\ra) =  \Ext(S(x,\eps), \R\la\eps,\delta,\gamma\ra),
\]
$D'$ is contained in $\Ext(S(x,\eps),\R\la \eps,\delta,\gamma\ra)$, 
and 
since 
\[
Z(P_0,\R\la \eps,\delta\ra) =  \Ext(S(x,\eps), \R\la\eps,\delta\ra),
\]
$\lim_\gamma(D')$ is in contained in $\Ext(S(x,\eps),\R\la \eps,\delta\ra)$. 

It now follows from the correctness of \cite[Algorithm 12.65]{BPRbookonline} (Parametrized Bounded Algebraic Sampling),
that
 the set of points associated to the 
$k$-univariate representations in $\mathcal{U}$ 
in 
Line~\ref{alg:line:4}
has a non-empty intersection with 
$\Ext(D,\R\langle \eps,\delta\rangle)
$  
for every semi-algebraically connected $D$ of $\RR(\sigma,\R\la\eps\ra^k) \cap S(x,\eps)$, 
$\sigma \in \SIGN_n(x,\mathcal{P})$.

    The correctness of  \cite[Algorithm 11.19 (Restricted Elimination)]{BPRbook2} ensures that 
    in 
    Line~ \ref{alg:line:8}
    the Thom encoding of the roots of $f$ and the signs of the polynomials in
    $\mathcal{P}$ at these roots is fixed on  any semi-algebraically connected component of the realization of a sign condition on $\mathcal{A}'(T,V)$ obtained by substituting $(T,V)$ in 
    place of $(\eps,\delta)$ in $\mathcal{A}(\eps,\delta)$. Define by $\alpha$ the sign condition on ${\mathcal A}'(T,V)$ satisfied at $(\eps,\delta)$ and by $C$ the semi-algebraically connected component of 
    $\RR(\alpha,\R^2)$ such that $(\eps,\delta)\in \Ext(C, \R\la \eps,\delta\ra)$.

In order to define a point in the extension of $C$ to $\R \langle \eps \rangle$ without changing the Thom encoding of the roots of $f$ and the signs of the polynomials in
    $\mathcal{P}$, we need to replace $\delta$ by  
     $c(\lambda) = \frac{\eval_\xi(a(\lambda))}{\eval_\xi(b(\lambda))}$
    which is small enough so that no element in ${\mathcal A}'(T,V)$ changes sign on the closed segment joining $(\eps,\delta)$  to 
   $(\eps,c(\lambda))$.
This is ensured by 
choosing $\lambda$ such that
$c(\lambda)= \min_{\lambda'\in \mathcal{A}}  c(\lambda')$ 
using the properties of the Cauchy bound
\cite[Lemma 10.7]{BPRbook2}.
This is accomplished in 
Line~\ref{alg:line:11}.

The correctness of 
the computation performed in Line~\ref{alg:line:13}
follows from the correctness of  \cite[Algorithm 10.13
    (Univariate Sign Determination)]{BPRbook2} and Remark~\ref{rem:xi}.
   
Now we need to obtain description of paths by  finding $t_0$ so that 
$(t_0,c(\lambda(t_0)))$ belongs to $C$, without  changing the Thom encoding, the signs of the polynomials in
    $\mathcal{P}$ and the choice of $h \in \mathcal{A}$. This is ensured by taking 
    \[
     t_0 = \frac{\eval_\xi(a_0)}{\eval_\xi(b_0)},
     \]
      where $a_0 = \min_{g\in \mathcal{B'}} a(g)$,  and $b_0= \max_{g\in \mathcal{B'}} b(g)$
      (in Line~\ref{alg:line:21}).
The correctness    of this step  follows from the correctness of  \cite[Algorithm 11.19 (Restricted Elimination)]{BPRbook2} and the properties of the Cauchy bound \cite[Lemma 10.7]{BPRbook2}, since the segment
    joining $(\eps,c(\lambda))$ to $(t_0,c(\lambda(t_0)))$ is entirely contained in $C$.
\end{proof}

\begin{proof}[Complexity analysis]
  The total number of $j \le k$-tuples examined is 
  \[
  \sum_{j \le k}
  \binom{s}{j} 4^{j}.
  \]
  Hence, the number of calls to \cite[Algorithm 12.65
  (Parametrized Bounded Algebraic Sampling)]{BPRbookonline} 
  is also bounded by
  $2 \sum_{j \le k} \binom{s}{j} 4^{j}$. Each such call costs $d^{O (k)}$
  arithmetic operations (addition, multiplication and computing signs) in $\D_\xi$. 
  Thus the total number of real 
  $k$-univariate representations $\mathcal{U}$ produced is
  bounded by 
  \[
  \sum_{j \le k} \binom{s}{j} 4^{j} O (d)^{k},
  \]
   while the number
  of arithmetic operations performed for computing sample points in $\R \la
  \eps , \delta \ra^{k}$, is bounded by 
  \[
  \sum_{j \le k} \binom{s}{j} 4^{j}
  d^{O (k)} =s^{k} d^{O (k)}.
  \]
  Using the complexity analysis of 
\cite[Algorithm 12.65
  (Parametrized Bounded Algebraic Sampling)]{BPRbookonline} 
  the degree in $T$ of
  the
  $k$-univariate representations output  in $\mathcal{U}$ is bounded by $N$,
  and their degrees in $\eps , \delta$ is bounded by
  $N'$,
  since $Q$ is a polynomial of degree in $X_1,\ldots,X_k$ bounded by $2d+4$ and the degrees in $\eps,\delta,\gamma$ bounded by 4.
 When $\D = \Z$, and the bitsize of the coefficients of the input polynomials
  is bounded by $\tau$, the bitsize of the 
  output 
  $k$-univariate representations in $\mathcal{U}$ is
  bounded by
  $$( \tau + \bit (s)) O(d)^{k+1}  
  \le \tau (d+k)^k O(d)^{k+1}
  ,$$ 
  since the
  total number $S$ of polynomials of degree $d$ in $k$ variables with bitsize
  bounded by $\tau$ satisfies
  \begin{eqnarray*}
    \bit(S)\leq (\tau+1) \binom{d+k}{k} \le & 2 \tau(d+k)^k . & 
  \end{eqnarray*}
  
  When we 
use   \cite[Algorithm 11.19 (Restricted Elimination)]{BPRbook2} with variable $U$, input $f$ and $\mathcal{P}_{u}$
  we output polynomials in $\mathcal{A}$ of degree in $\eps,\delta$ bounded by $2dNN'=O(d)^{2k+2}$ 
     using the complexity analysis of 
   \cite[Algorithm 11.19 
    (Restricted Elimination)]{BPRbook2}.
    Moreover, when $\D=\Z$ and $x=0$, their  bitsize  is bounded by
     $ \tau (d+k)^k O(d)^{2k+1}. 
     $ 
    The sign determination takes 
  \[
  s \sum_{j \le k}
  \binom{s}{j} 4^{j} d^{O (k)} =s^{k+1} d^{O (k)}
  \] 
  arithmetic operations in $\D_\xi$, using the complexity analysis of 
  \cite[Algorithm 10.13
    (Univariate Sign Determination)]{BPRbook2}.
    After subsituting $\delta$ by  $c(\lambda)=
    \min_{\lambda'  \in \mathcal{A}}  c (\lambda')$ in $\mathcal{U}$, 
    we obtain
 $\mathcal{V}$ of degree   $2dNN'^2=O(d)^{3k+3}$  in $\eps$ 
  and (when $\D=\Z$ and $x=0$) of bitsize 
     $ \tau (d+k)^k O(d)^{3k+2}
     $.
     
   After computing  $\mathcal{B}
   \subset \D_\xi [ \eps ]$  and taking 
 \[
 \mathcal{B'}=\mathcal{B} \cup \{a(\lambda) b(\lambda')-a (\lambda')b(\lambda),  
\lambda'\in \mathcal{A}\setminus \{\lambda \}\},
\] 
    we obtain  $\mathcal{B'}$ of
of degree
    $2dN N'^2=O(d)^{3k+3}$ in  $\eps$ . Moreover  when $\D=\Z$ and $x=0$, the  bitsize  of elements of $\mathcal{B'}$ is
     $ \tau (d+k)^k O(d)^{3k+2}$, and,
      with  $a_0 = \min_{g\in \mathcal{B'}}  a(g)$,  and $b_0= \max_{g\in \mathcal{B'}} b(g)$,
     \[
      t_0 = \frac{\eval_\xi(a_0)}{\eval_\xi(b_0)},
      \]
      is a rational number with numerator and denominator of bitsize
       $ \tau (d+k)^k O(d)^{3k+2}
       $.
       
       The final complexity bound follows from the fact that  arithmetic operations (addition, multiplication and computing signs) in $\D_\xi$ cost $d'^{O(1)}$ arithmetic operations in $\D$.
\end{proof}

\begin{remark}
An improvement of \cite[Algorithm 12.65]{BPRbookonline} (Parametrized Bounded Algebraic Sampling) with better degree bounds with respect to the main variable and the parameters would 
immediately 
improve our results.
\end{remark}
   
 \begin{proof}[Proof of Theorem    \ref{qcsltotcc}]
 Theorem \ref{qcsltotcc} is an immediate consequence of the proof of correctness and the complexity analysis of Algorithm
 \ref{alg:1}.
 \end{proof}
 
 \begin{proof}[Proof of Theorem \ref{qcsl} (Quantitative Curve Selection Lemma)]
 If 
$x\in S$,
 just pick the semi-algebraic path constantly equal to $x$.
If $x\in \bar S\setminus S$, there exists a sign condition  $\sigma \in \SIGN_n(x,\mathcal{P})$ such that $\RR(\sigma,\R^k)$ is contained in $S$. Such a $\sigma$ can be obtained by examining the description of $S$ as a $\mathcal{P}$ semi-algebraic set and the list $\SIGN_n(x,\mathcal{P})$. So  Theorem 
\ref{qcsl} 
is an immediate consequence of   Theorem \ref{qcsltotcc} without increasing the complexity.
\end{proof}

\section{Consequences}

\subsection{Degree of the Zariski closure of the image of the little path}

We now prove that the 
degree of the Zariski closure of the image of each path, $\varphi_\sigma, \sigma \in \SIGN_n(x,\mathcal{P})$ in the output 
of Algorithm \ref{alg:1},  is bounded by  
$4dN^2 N'^2=O(d)^{4k+3}$.

In order to obtain this bound we need to bound
the number of intersection points with an hyperplane $H$ of equation
$$a_0+a_1X_1+\ldots+a_kX_k=0.$$ 
Eliminating $U$ from
$f(T,U)$ and $a_0 g_0(T,U)+a_1 g_1(T,U)+\ldots+a_k g_k(T,U)$
gives a polynomial in $T$ of degree $2 N \times 2dNN'^2=4dN^2 N'^2=O(d)^{4k+3}$.

So we have proved the following.

\begin{theorem}[Degree of the Zariski closure of a little path]
For every 
\[
\sigma \in \SIGN_n(x,\mathcal{P}),
\] 
there exists 
$t_0>0$  in $\R$, and a semi-algebraic path $\varphi_\sigma:[0, t_0)\rightarrow \R^k$ such that 
\begin{enumerate}[(a)]
\item
$\varphi_\sigma (0) = x$ and $\varphi_\sigma ((0, t_0)) \subset \RR(\sigma,\R^k)$;
\item
the degree of the Zariski closure of the image of $\varphi$ is bounded by 
\[
4dN^2 N'^2=O(d)^{4k+3}.
\] 
\end{enumerate}
In particular if $S$ is a $\mathcal{P}$-semi-algebraic set and $x\in \bar S$, 
 there exists 
$t_0>0$  in $\R$, and a semi-algebraic path $\varphi: [0, t_0) \rightarrow \R^k$, such that $\varphi (0) = x$ and $\varphi ((0, t_0)) \subset S$,
and the degree of the Zariski closure of the image of $\varphi$  is bounded by $4dN^2 N'^2=O(d)^{4k+3}$.
\end{theorem}

This is an improvement on the bound $d^k((d-1)^k+2)^{k-1}=d^{O(k^2)}$ proved in \cite[Proposition 6.2]{Kurdyka2014}.

\subsection{Finding the 
real isolated points
of a semi-algebraic set}

We now use the construction of the little paths to study zero-dimensional semi-algebraically connected components of semi-algebraic sets. We prove the following result.

\begin{theorem}[Computation of 
real isolated points
 of a semi-algebraic set]
\label{thm:zero-dim}
  Let 
   $\mathcal{P} \subset \D[X_1,\ldots,X_k]$ be a finite set of $s$ polynomials of maximum 
   total
   degree $d$,
  $S$ a $\mathcal{P}$-semi-algebraic set.
The real isolated points 
of $S$ are associated to real univariate representations that
can be computed using
$s^{2 k+1} d^{O(k)}$ 
arithmetic operations in $\D$.
\end{theorem}

Given $x\in \bar S$, we denote by $\SIGN_n (x,S)$ the sign conditions of 
$\SIGN_n (x,\mathcal{P})$ whose  realization is contained in $S$.

\begin{proof}[Proof of Theorem~\ref{thm:zero-dim}]
First compute a set of real univariate representations whose set of associated points has a non-empty intersection with  every semi-algebraically connected component of $S$ using \cite[Algorithm 13.2 (Sampling)]{BPRbook2}. 
This step uses  $s^{k+1} d^{O(k)}$ arithmetic operations in $\D$, and  produces
 $s^{k} d^{O(k)}$ real univariate representations of degrees bounded by $d' = O(d)^k$.
For each of these real univariate representations, $\xi$, with associated point $x$, 
call Algorithm~\ref{alg:1} with $\xi$ and $\mathcal{P}$ as input and obtain
$\SIGN_n(x,\mathcal{P})$, and compute the subset $\SIGN_n (x,S) \subset \SIGN_n(x,\mathcal{P})$.
The set of real isolated points  of $S$ consists of those points $x$ computed in the previous step 
for which the set $\SIGN_n (x,S)$ is empty.

Since each call to Algorithm~\ref{alg:1} uses $d'^{O(1)} s^{k+1} d^{O(k)} = s^{k+1} d^{O(k)}$ 
arithmetic operations in $\D$,
the total complexity of the algorithm described above is bounded by 
\[
s^{2 k+1} d^{O(k)}.
\] 
\end{proof}

Note that Theorem \ref{thm:zero-dim}  gives a test for a semi-algebraic set being zero-dimensional,
generalizing
\cite[Theorem 13.17]{BPRbook2}.

In the special case of an algebraic set, this is an improvement on the bound
$(kd)^{O(k \log k)}$ proved in \cite{LeMohabTimo}. 
In contrast to \cite{LeMohabTimo}, our approach using the little paths avoids using  more sophisticated techniques 
from algorithmic semi-algebraic geometry such as computing roadmaps. 

\section{Acknowledgements}

We would like to thank
 Krzysztof Kurdyka for attracting our attention to 
the question of quantitative curve selection lemma.
We are grateful to the anonymous referee for relevant remarks and suggestions.

\bibliographystyle{amsplain}
\bibliography{../../BR-aux/master}

\end{document}